\documentclass[10pt]{amsart}

\long\def\comment#1{}

\makeatletter
\def\imod#1{\allowbreak\mkern10mu({\operator@font mod}\,\,#1)}
\makeatother

\newtheorem{theo}{{\bf Theorem}}[section]
\newtheorem{defn}[theo]{{\bf Definition}}
\newtheorem{cor}[theo]{{\bf Corollary}}
\newtheorem{lem}[theo]{{\bf Lemma}}
\newtheorem{prop}[theo]{{\bf Proposition}}
\newtheorem{rem}{Remark}

\title[Third moment of the remainder term for Heisenberg
manifolds]{Third moment of the remainder term in Weyl's law for
Heisenberg manifolds}

\author{Mahta Khosravi}
\address{ School of Mathematics, Institute for Advanced Study, Princeton,
NJ 08540, USA} \email{khosravi@math.ias.edu; \!\!\!\! {\it Current E-mail address:}\,\,\,khosravi@math.jhu.edu}

\thanks{The author was supported by an NSERC postdoctoral fellowship
and partially by an NSF grant DMS 0111298 through
the Institute for advanced Study.}

\begin{document}


\begin{abstract}
Let $R(t)$ be the remainder term in Weyl's law for a 3-dimensional
Riemannian Heisenberg manifold with a certain {\lq arithmetic\rq}
metric. We prove a third moment result stating that $ \int_1^T
R(t)^3 dt =d_3\, T^{{13}/{4}}+O_\delta(T^{45/14+\delta}) $, where
$d_3$ is a specific positive constant which can be evaluated
explicitly. This proves the asymmetric behavior of $R(t)$ about
the $t$-axis. This result is consistent with the conjecture of
Petridis and Toth stating that
$R(t)=O_{\delta}(t^{{3}/{4}+\delta})$. Similar results hold for
\mbox{2$n$+1}-dimensional Heisenberg manifolds with arithmetic
metrics.

\end{abstract}

\maketitle
\section{Introduction}

Let $(M, g)$ be a closed $n$-dimensional Riemannian manifold with
metric $g$ and Laplace-Beltrami operator $\Delta$. We denote its
spectral counting function by $N(t)$, defined as the number of the
eigenvalues of $\Delta$ not exceeding $t$. A celebrated theorem of
H\"{o}rmander \cite{Ho} asserts that as $t\rightarrow \infty$,

\begin{eqnarray}
N(t)=\frac{\hbox{vol}(B_n) \hbox{vol}(M)}{(2\pi)^n}t^{{n}/{2}} +
O(t^{{(n-1)}/{2}}),
\end{eqnarray}
where $\hbox{vol}(B_n)$ is the volume of the $n$-dimensional unit
ball.

By considering the unit sphere, it is straightforward to show that
the estimate for the remainder term in H\"{o}rmander's theorem
defined by
\begin{eqnarray}
R(t)=N(t)-\frac{\hbox{vol}(B_n)
\hbox{vol}(M)}{(2\pi)^n}t^{{n}/{2}},
\end{eqnarray}
is in general sharp. However, the question of determining the
optimal bound for this remainder term for any given manifold is a
difficult one and depends on the properties of the associated
geodesic flow. In many cases, this is an open problem.
Nevertheless, for certain types of manifolds some improvements
have been obtained and in a few cases the conjectured optimal
bound has been attained (see \cite{BG}, \cite{Be}, \cite{Bl},
\cite{Fr}, \cite{Go}, \cite{Hu}, \cite{Iv}, \cite{KP} and
\cite{Vo}).

\comment{The results obtained in this direction can be classified
in three categories: (i) The first type of results deal with the
upper bound for the rate of the growth of the remainder term( i.e.
the $O$-results). (ii) The second type deal with finding a lower
bound for this growth( i.e. the $\Omega$-results). (iii) Finally,
the third type are results about the averages and the moments of
the remainder term.}

The results obtained in this direction can be separated into two
categories: (i) upper and lower bounds for the rate of growth of
the remainder term (i.e. the $O$-results and $\Omega$-results
respectively); (ii) the distribution of the remainder term about
the $t$-axis and averages and moments of the remainder term.

In this article, we address a result of type (ii) on Heisenberg
manifolds. Following our previous work in \cite{KT} where we
evaluated the second moment of the remainder term in Weyl's law
for Heisenberg manifolds, we investigate the third moment of this
remainder term.

We first review some well known results. For manifolds with
completely integrable geodesic flows satisfying some clean          
intersection hypothesis, Duistermaat and Guillemin \cite{DG} have
proven that $R(t)=o(t^{(n-1)/2}).$ For generic convex surfaces of
revolution, Colin de Verdi\`{e}re \cite{Co} showed that
$R(t)=O(t^{1/3}).$ The simplest compact manifold with integrable
geodesic flow is the 2-torus ${\mathbb{T}}^2$.  Hardy's conjecture
for ${\mathbb{T}}^2$ \cite{Ha} asserts that

\begin{displaymath}
R(t)=O_{\delta}(t^{\frac{1}{4}+\delta}),
\end{displaymath}
where here, and hereafter in this article, $\delta$ is any
arbitrary small positive number and the $O_\delta$ notation
indicates the implied constant may depend on the value of
$\delta$. Hardy further proved that for ${\mathbb{T}}^2$ this is
the best possible upper bound. To be more precise, he proved the
following lower bound results:
\begin{displaymath}
R(t)=\Omega_{-}((t\log t)^{\frac{1}{4}}) \;\; \hbox{ and
}\;\; R(t)=\Omega_{+}(t^{\frac{1}{4}}).
\end{displaymath}
These lower bound results have since been improved and the best
known result today is due to Soundarajan \cite{Sou} who proved
that
\begin{displaymath}
R(t)=\Omega\left((t\log t)^{1/4}(\log_2
t)^{(3/4)(2^{1/3}-1)}(\log_3 t)^{-5/8}\right),
\end{displaymath}
where $\log_2 t=\log\log t$ and $\log_3 t = \log\log_2 t$ . Moving
to the moment results on flat tori, there is a classical result of
Cram\'{e}r \cite{Cr} which states that for ${\mathbb{T}}^2$

\begin{displaymath}
\int_1^T {R(t)}^2 dt= c_2\, T^{\frac{3}{2}}+
O_\delta(T^{\frac{5}{4}+\delta}),
\end{displaymath}
as $T\rightarrow \infty$ where $c_2=\frac{1}{6\pi^3}\sum_1^\infty
\frac{r(n)^2}{n^{3/2}}$ with $r(n)=\#\{(a, b)\in {\mathbb{Z}}^2;
n=a^2+b^2\}$. This result is consistent with Hardy's conjecture.
Tsang \cite{Ts} has evaluated the third and fourth moments of the
remainder term of Weyl's law on flat tori proving that for some
specific negative constant $c_3$ and positive constant $c_4$
\begin{displaymath}
\int_1^T {R(t)}^3 dt= c_3\,
T^{\frac{7}{4}}+O(T^{\frac{7}{4}-\epsilon}),
\end{displaymath}and
\begin{displaymath}
\int_1^T {R(t)}^4 dt= c_4\, T^2+O(T^{2-\epsilon}),
\end{displaymath}
as $T\rightarrow \infty$ for some $\epsilon >0$. The fifth moment
result on flat tori is due to the author \cite{K} who has recently
proven that
$$\int_1^T R(t)^5 dt =c_5
\:T^{\frac{9}{4}}+O_\delta(T^{\frac{727}{324}+\delta}), $$ as
$T\rightarrow \infty$ where $c_5$ is a specific negative constant
.

From the work of Heath-Brown \cite{HB} in 1992, we know that the
normalized remainder term $t^{-1/4} R(t)$ has an asymptotic
distribution function in the sense that for any interval $I$
$$ T^{-1} \hbox{mes}\left\{t\in[1, T], t^{-\frac{1}{4}} R(t)\in I\right\}
\longrightarrow \int _I f(\alpha) d \alpha $$ as $T\rightarrow
\infty$. He showed that the density function and its derivatives
decay on the real line faster than exponentially. His methods also
show the convergence of the moments up to order nine even though           
they are not strong enough to provide the rate of convergence.

As the first, natural, non-commutative generalization of
${\mathbb{T}}^2$ consider 3-dimensional Heisenberg manifolds
$(\Gamma\backslash H_1, g)$. These manifolds have completely
integrable geodesic flows \cite{Bu}. Petridis and Toth \cite{PT}
proved that for certain {\lq arithmetic\rq} Heisenberg metrics
$R(t)=O_\delta(t^{5/6+\delta}).$
Later in \cite{PT2} the exponent was improved to
$R(t)=O_\delta(t^{34/41+\delta})$ and the result extended to all
left-invariant Heisenberg metrics. It was conjectured in \cite{PT}
that for $(\Gamma\backslash H_1, g)$,
\begin{equation}\label{conjecture}
R(t)=O_{\delta}(t^{\frac{3}{4}+\delta}).
\end{equation}
Moreover, as evidence for this conjecture, Petridis and Toth
\cite{PT} proved the following $L^2$-result for $(\Gamma\backslash
H_1, g)$ with the arithmetic metric by averaging locally over the
moduli space of left-invariant metrics
\begin{displaymath}
\int_{I^3}|N(t;
u)-\frac{1}{6{\pi}^2}\hbox{vol}(M(u))t^{\frac{3}{2}}|^2 du\leq
C_{\delta} t^{\frac{3}{2}+\delta},
\end{displaymath}
where $I=[1-\epsilon, 1+\epsilon]$. They also proved that for
sufficiently large $T$,
\begin{displaymath}
\frac{1}{T}\int_{T}^{2T}|N(t)-\frac{1}{6{\pi}^2}\hbox{vol}(M)t^{\frac{3}{2}}|dt\gg
T^{\frac{3}{4}}.
\end{displaymath}
It has been noted that the conjecture (\ref{conjecture}) follows
from the standard conjectures on the growth of exponential sums,
see \cite{PT2}.

In higher dimensions, i.e. $(\Gamma\backslash H_n, g)$ where
$n>1$, in joint work with Petridis \cite{KP} we proved that for
generic irrational metrics
\begin{displaymath}
R(t)=O_{\delta}(t^{n-\frac{1}{4}+\delta}).
\end{displaymath}
Moreover, we demonstrated that this bound is sharp.

As evidence for (\ref{conjecture}), we proved with Toth \cite{KT}
the $L_2$-result

\begin{eqnarray}
\int_1^T {R(t)}^2 dt =d_2\,
T^{\frac{5}{2}}+O_{\delta}(T^{\frac{9}{4}+\delta})\label{R},
\end{eqnarray}
where $d_2$ is an explicitly evaluated positive constant.

The main purpose of this paper is to prove that $\int_1^T {R(t)}^3
dt$ similarly has meaningful asymptotics for $(\Gamma\backslash
H_n, g)$.
\medskip





\begin{theo}\label{thm} For $(2n+1)$-dimensional
Heisenberg manifold with the metric $g=\begin{pmatrix}
  I_{2n\times 2n} & 0\\
  0 & 2\pi
\end{pmatrix}$, where $I_{2n\times 2n}$ is the identity
 matrix, there exists a positive constant $d_3$ such that
\begin{eqnarray}
\int_1^T {R(t)}^3 dt =d_3\,
T^{3n+\frac{1}{4}}+O_\delta(T^{3n+\frac{3}{14}+\delta}).
\end{eqnarray}

\end{theo}

\begin{rem} Without loss of generality, we prove
Theorem \ref{thm} only for 3-dimensional Heisenberg manifolds. In
general dimensions the proof follows in an identical manner.  See \cite{KT} for the exponential representation of the remainder
term of Weyl's law on higher-dimensional Heisenberg manifolds.
\end{rem}

\begin{rem} Theorem \ref{thm} also holds for rational
$(2n+1)$-dimensional Heisenberg manifolds (for the definition of
rationality refer to \cite{KP}). However in this case, our methods
to prove the positivity of the constant do not apply any more. In
the case of irrational Heisenberg manifolds we do not currently
know how to prove the result.
\end{rem}

\begin{rem}
Based on the the method of the proof, which implies a large
truncation index in the summation defining the mollified remainder
term, we are not able to modify this method to prove a 4th moment
result.
\end{rem}

\section{Background on Heisenberg manifolds}

We review here some of the basic properties of Heisenberg manifolds. The
reader should consult \cite{GW}, \cite{St} or \cite{Fo} for
further details.

\subsection{Basic definitions and notation} For any two real numbers
$x$ and $y$ let
\begin{displaymath}
\gamma(x, y, t)=\begin{pmatrix}
  1 & x & t \\
  0 & 1 & y \\
  0 & 0 & 1
\end{pmatrix}  \; , \;
 X(x, y, t)=\begin{pmatrix}
  0 & x & t \\
  0 & 0 & y \\
  0 & 0 & 0
\end{pmatrix}.
\end{displaymath}
The real $3$-dimensional Heisenberg group $H_1$ is the Lie
subgroup of $Gl_3(\mathbb{R})$ consisting of all matrices of the
form $\gamma(x, y, t)$:

\begin{displaymath}
H_1=\{\gamma(x, y, t): x,y\in {\mathbb{R}} , t\in \mathbb{R} \}.
\end{displaymath}
The Lie algebra of $H_1$ is:

\begin{displaymath}
{\frak{h}}_1=\{X(x, y, t): x,y\in {\mathbb{R}} , t\in
\mathbb{R}\}.
\end{displaymath}
The matrix exponential maps ${\frak{h}}_1$ diffeomorphically onto
$H_1$ and is given by the formula

\begin{displaymath}
  \begin{cases}
    \exp: {\frak{h}}_1 \mapsto H_1 , \\
    X(x, y, t)\mapsto \gamma(x, y, t+\frac{1}{2}x.y).
  \end{cases}
\end{displaymath}
The product operation in $H_1$ and Lie bracket in ${\frak{h}}_1$
are given by

\begin{displaymath}
\gamma(x, y, t).\gamma(x', y', t')= \gamma(x+x', y+y', t+t'+x.y'),
\end{displaymath}
\begin{displaymath}
[X(x, y, t), X(x', y', t')]= X(0, 0, x.y'-x'.y).
\end{displaymath}
The algebra ${\frak{z}}_1=\{X(0, 0, t), t\in \mathbb{R}\}$ is both
the center and the derived subalgebra of ${\frak{h}}_1$. It is
also convenient to identify the subspace $\{X(x, y, 0), x,y\in
{\mathbb{R}}\}$ of ${\frak{h}}_1$ with ${\mathbb{R}}^{2}$ and so,
${\frak{h}}_1={\mathbb{R}}^{2}\oplus {\frak{z}}_1$.

The standard basis of ${\frak{h}}_1$ is the set $\delta=\{X_1,
Y_1, Z\}$, where the first $2$ elements are the standard basis of
${\mathbb{R}}^{2}$ and $Z=X(0, 0, 1)$. The only nonzero bracket
among the elements of $\delta$ is given by $[X_1, Y_1]=Z$.

\medskip

\begin{defn} A Riemannian Heisenberg manifold is a pair $(\Gamma\backslash H_1, g)$
where $\Gamma$ is a uniform discrete subgroup of $H_1$ ({\lq
uniform\rq} means that the quotient $\Gamma\backslash H_1$ is
compact), and $g$ is a Riemannian metric on $\Gamma\backslash H_1$
whose lift to $H_1$ is left $H_1$-invariant.
\end{defn}

\subsection{Classification of the uniform discrete subgroups of $H_1$} For
every positive integer $r$, define
\begin{displaymath}
\Gamma_r=\{\gamma(x, y, t): x\in r {\mathbb{Z}}, y\in
{\mathbb{Z}}, t\in \mathbb{Z}\}.
\end{displaymath}
It is clear that $\Gamma_r$ is a uniform discrete subgroup of
$H_1$.

\begin{theo} ( \cite{GW}, Theorem 2.4) The subgroups $\Gamma_r$ classify
the uniform discrete subgroups of $H_1$ up to
automorphisms. In other words, for every uniform discrete subgroup
of $H_1$ there exists a unique $r\in {\mathbb{Z}}_+ $ and an
automorphism of $H_1$ which maps $\Gamma$ to $\Gamma_r$. Also, if
two subgroups $\Gamma_r$ and $\Gamma_s$ are isomorphic then $r$
and $s$ are equal.
\end{theo}

\medskip

\begin{cor} ( \cite{GW}, Corollary 2.5) Given any Riemannian Heisenberg manifold
$M=(\Gamma\backslash H_1, g)$, there exists a unique positive
integer $r$ and a left-invariant metric $\tilde{g}$ on $H_1$ such
that $M$ is isometric to $(\Gamma_r\backslash H_1, \tilde{g})$.
\end{cor}

\medskip

Since every left-invariant metric $g$ on $H_1$ is uniquely
determined by an inner product on ${\frak{h}}_1$, the
left-invariant metrics can be identified with their matrices
relative to the standard basis of ${\frak{h}}_1$. For any $g$ we
can choose an inner automorphism $\varphi$ of $H_1$ such that
${\mathbb{R}}^{2}$ is orthogonal to ${{\frak{z}}}_1$ with respect
to ${\varphi}^*g$. Therefore, $(\Gamma\backslash H_1, g)$ will be
isometric to $(\Gamma\backslash H_1, {\varphi}^*g)$ and we can
replace every left-invariant metric $g$ by ${\varphi}^*g$ and
always assume that the metric $g$ has the form $g=\begin{pmatrix}
  h & 0 \\
  0 & g_{3}
\end{pmatrix},$
where $h$ is a positive-definite $2\times 2$ matrix and $g_{3}$ is
a positive real number. The volume of the Heisenberg manifold is
given by the formula $\hbox{\hbox{vol}}( \Gamma_r\backslash H_1,
g)=r\sqrt{\hbox{det}(g)}$.

\subsection{The spectrum of Heisenberg manifolds.}
Let $M=( \Gamma\backslash H_1, g)$ be a Heisenberg manifold where
the metric $g$ is in the arithmetic form $g=\begin{pmatrix}
  I_2 & 0\\
  0 & 2\pi
\end{pmatrix}$ and $I_2$ is the two by two identity matrix.

Let $\Sigma$ be the spectrum of the Laplacian on $M=(
\Gamma\backslash H_1, g)$, where the eigenvalues are counted with
multiplicities. Then, $\Sigma ={\Sigma}_1\cup {\Sigma}_2$ ( see
\cite{GW} page 258) where,

\begin{displaymath}
{\Sigma}_1=\{\lambda(m, n)=4\pi^2(m^2+n^2);  (m, n)\in
{\mathbb{Z}}^{2} \},
\end{displaymath}
such that $\lambda(m, n)$ is counted once for each pair $(m, n)\in
{\mathbb{Z}}^{2}$ such that $\lambda=\lambda(m, n)$.

\noindent The second part of the spectrum, ${\Sigma}_2$, is the
set

\begin{displaymath}
{\Sigma}_2=\{\mu(c,k)=2\pi c(c+(2k+1)); c\in {\mathbb{Z}}_+,
k\in({\mathbb{Z}}_+\cup\{0\}) \},
\end{displaymath}
where every $ \mu(c,k) $ is counted with multiplicity $2c$.


\section{Estimates for regularized spectral counting function}

The idea of the proof of theorem ~(\ref{thm}) is to use the
exponential sum representation which we proved for the regularized spectral counting
function in \cite{KT} and apply a modified version of the
method used by Tsang \cite{Ts}.

In this section we give a short overview on some of the notation and
results proved in \cite{KT}. Let $N(t)$ to be the spectral counting
function defined by

\begin{eqnarray}
N(t)=N_T(t)+N_H(t),
\end{eqnarray}
where $N_T(t)$ is the spectral counting function of the torus,
defined by

\begin{displaymath}
N_T(t)=\#\{\lambda\in\Sigma_1 ; \lambda\leq t\},
\end{displaymath}
and $N_H(t)$ is defined by

\begin{displaymath}
N_H(t)=\# \{\lambda\in\Sigma_2 ; \; \lambda\leq t \}.
\end{displaymath}
The estimates for $N_T(t)$ are well-known. For example,
\begin{eqnarray}
N_T(t)=\frac{t}{4\pi}+O(t^{\frac{1}{2}}),
\end{eqnarray}
will suffice for our purposes. This bound was known to Gauss. To
evaluate $N_H(t)$, we write

\begin{eqnarray}
N_H(t)=\sum_{c(c+(2k+1))\leq t/2\pi}2c.\label{n}
\end{eqnarray}
Define $A_t=\{(x, y); x>0, y>0, x(x+2y+1)\leq t\}$. Then, we have

\begin{eqnarray}
N_H(2\pi t)=\sum_{(c,k)\in {\mathbb{Z}}^2}(2c)\chi_{A_t}(c,
k).\label{n-21}
\end{eqnarray}

To obtain an exponential-sum representation for the remainder
term we need to apply the Poisson summation formula to write the
remainder term, corresponding to type II eigenvalues, in a form
which can be estimated by the method of the stationary phase.

However, to justify the application of the Poisson summation
formula for $N_H(2\pi t)$, we need to regularize the
characteristic function $\chi_{A_t}$. Take $\rho$ to be a smooth
symmetric positive function on ${\mathbb{R}}^2$ with
$\int_{{\mathbb{R}}^2}\rho(x, y)dxdy=1$ and
$\hbox{supp}(\rho)\subseteq[-1, 1]^2$. Let $\rho_\epsilon(x,
y)=\epsilon^{-2}\rho({x}/{\epsilon}, {y}/{\epsilon})$,
where we make an explicit choice of $\epsilon>0$ later on.
Consider the mollified counting functions

\begin{eqnarray}
N^{\epsilon} _H(t):=\sum_{(c,k)\in
{\mathbb{Z}}^2}(2c)\chi_{A_t}(c, k) \ast \rho_\epsilon(c,
k).\label{n-eps1}
\end{eqnarray}

\begin{lem} \label{lem1}Let $T$ be an arbitrarily large number and
put $\epsilon=T^{-\gamma}$ for an arbitrary fixed $\gamma\in (0,1]$.
Then, for $1<t<T$ and a constant $c_\gamma>2$ which depends only on
$\gamma$, we have

\begin{displaymath}
N^{\epsilon}_H(t-c_\gamma T^{1-\gamma})\leq N_H(2\pi t)\leq
N^{\epsilon}_H (t+c_\gamma T^{1-\gamma}).
\end{displaymath}

\end{lem}

\begin{proof} We prove the first inequality in \ref{lem1}.
The other inequality follows in the same way. Given $A_t=\{(x, y);
x>0, y>0, x(x+y)\leq t\}$, let $\partial A_t$ to be the hyperbola
$x(x+y)=t$. If a point $X=(x, y)\in{{\mathbb{Z}}_+}^2$ lies at a
distance greater than $\sqrt{2}\epsilon$ from $\partial A_t$, then
$\chi_{A_t}\ast \rho_\epsilon(X)=\chi_{A_t}(X)$.

Therefore, by taking $\Omega_1=\{(c,k)\in {\mathbb{Z}}^2;
\hbox{dist}((c,k),
\partial A_{t+K\epsilon})>\sqrt{2}\epsilon\}$, we have,

\begin{eqnarray*}
N^{\epsilon} _H(t+K\epsilon)&=&\sum_{(c,k)\in
{\mathbb{Z}}^2}(2c)(\chi_{A_{t+K\epsilon}}\ast \rho_\epsilon)(c,
k)\\&=&\sum_{ (c,k)\in\Omega_1}(2c)\chi_{A_{t+K\epsilon}}(c,
k)+\sum_{(c,k)\in{\mathbb{Z}}^2\setminus\Omega_1}(2c)(\chi_{A_{t+K\epsilon}}\ast
\rho_\epsilon)(c, k).
\end{eqnarray*}
On the other hand,
\begin{eqnarray*}
N _H(2\pi t)=\sum_{(c,k)\in {\mathbb{Z}}^2}(2c)\chi_{A_t}(c, k).
\end{eqnarray*}
So, to get $N^{\epsilon} _H(t+K\epsilon)\geq N _H(2\pi t)$, it
suffices to choose $\epsilon$ and $K$ so that ${\mathbb{Z}}^2 \cap
A_t \subseteq \Omega_1$. Since the closest point of
${\mathbb{Z}}^2 \cap A_t$ to $\partial A_{t+K\epsilon}$ is $(1,
[t-1])$, it suffices to require that

\begin{eqnarray}
\hbox{dist}((1,t), (\frac{-t+\sqrt{t^2+4t+4K\epsilon}}{2},
t))>\sqrt{2}\epsilon. \label{lemma}
\end{eqnarray}

Equation ~(\ref{lemma}) is equivalent to $4K\epsilon>4{\epsilon}^2
+4+4\epsilon t+8\epsilon$. So, it is enough to choose $K= 2T$ and
$\epsilon= T^{-\gamma}$. The inequality $N^{\epsilon}_A(t-c_\gamma
T^{1-\gamma})\leq N _H(2\pi t)$ can be proved in the same way.
\end{proof}

\begin{rem}\label{rem2}
Lemma \ref{lem1} will help us to convert our average results on
$N^{\epsilon} _H(t)$ back to $N_H(t)$. However, for this
conversion we need $\gamma>{3}/{4}$.\end{rem}

\begin{rem}
Based on the condition $\gamma>{3}/{4}$, which implies a large
truncation index in the summation defining $R_H^\epsilon(t)$,
we are not able to modify this method to prove a 4th moment result.
\end{rem}
\medskip

\begin{rem}\label{rem3}\hspace{.7in}

\begin{enumerate}\item Henceforth, we always assume $\epsilon=T^{-\gamma}$ for a
fixed large $T$, fixed $\gamma\in(0,1]$ and $t\in[1, T]$. Also we
assume that $\delta$ is an arbitrary small positive number
independent of $T$. \item By the notation $f(x)\ll g(x)$, we mean
that there exists a positive constant $C$ such that $|f(x)|\leq C
|g(x)|$ for every $x$.\end{enumerate}
\end{rem}
\medskip

\begin{prop}\label{prop1}(\cite{KT})The following asymptotic
expansion holds for $N^{\epsilon} _H$:

\begin{eqnarray}
N^{\epsilon} _H
(t)=\frac{2}{3}t^{\frac{3}{2}}-\frac{1}{2}t+R^{\epsilon}_H(t)
+O(t^{\frac{1}{2}+\delta}),\label{N-eps1}
\end{eqnarray}
where,
\begin{eqnarray}
R^{\epsilon}_H(t)&=&\frac{t^{\frac{3}{4}}}{\pi}\sum_{\stackrel{0<\nu<\mu,}{
\mu\equiv\nu \imod{2}}} (-1)^\nu
\cos(2\pi\sqrt{t}\sqrt{\mu\nu}-\frac{\pi}{4})
\mu^{-\frac{5}{4}}\nu^{-\frac{1}{4}}
\widehat{\rho}_\epsilon(\frac{\mu+\nu}{2},\nu)
\nonumber\\
&+&\frac{t^{\frac{3}{4}}}{2\pi}\sum_{0<\nu}(-1)^\nu
\cos(2\pi\sqrt{t}\nu-\frac{\pi}{4})\nu^{-\frac{3}{2}}
\widehat{\rho}_\epsilon(\nu, \nu)\label{repeat}.
\end{eqnarray}
\end{prop}

\vspace{.1in}

\section{Proof of theorem \ref{thm}}
Given the formula for the regularized counting function in
Proposition \ref{prop1}, we prove Theorem \ref{thm} in three
steps: First, we truncate the exponential sum representing
$R_H^{\epsilon}$ at a suitable term. Then, we apply a modified
version of Tsang's method to this truncated sum.
 Finally, using Lemma
\ref{lem1}, we eliminate the mollifier $\rho_\epsilon$ and prove Theorem
\ref{thm}.

\begin{lem}\label{lem} Let $F^{\epsilon}_H(t)$ be the first summation on the
right-hand side of ~(\ref{repeat}), then we have
\begin{eqnarray}
F^{\epsilon}_H(t)=\sum_{\stackrel{0<\nu<\mu ;
\mu\nu<T^{\alpha}}{\mu\equiv\nu \imod{2}}}(-1)^\nu t^{\frac{3}{4}}
\cos(2\pi\sqrt{t}\sqrt{\mu\nu}-\frac{\pi}{4})
\mu^{-\frac{5}{4}}\nu^{-\frac{1}{4}}
\widehat{\rho}_\epsilon(\frac{\mu+\nu}{2},\nu)
+O(T^{1/2}),\nonumber
\end{eqnarray}
where $\alpha$ is an arbitrary positive number lying in $(2\gamma,
2)$.
\end{lem}

\begin{proof}Since $\widehat{\rho}$ is a Schwartz function, for any
positive integer $m$ we have
\begin{eqnarray}
\widehat{\rho}_\epsilon(\frac{\mu+\nu}{2},\nu)<<\frac{1}{{(\epsilon^2\mu\nu)}^m},
\end{eqnarray}for $\mu>\nu>0$.
Applying $\epsilon=T^{-\gamma}$ and letting $\alpha>0$ we then have
\begin{equation*}
\sum_{0<\nu<\mu ; \mu\nu\geq T^{\alpha}} (-1)^\nu t^{\frac{3}{4}}
\cos(2\pi\sqrt{t}\sqrt{\mu\nu}-\frac{\pi}{4})
\mu^{-\frac{5}{4}}\nu^{-\frac{1}{4}}
\widehat{\rho}_\epsilon(\frac{\mu+\nu}{2},\nu)\end{equation*}\begin{eqnarray}\ll
T^{\frac{3}{4}}\sum_{k\geq T^{\alpha}}
 k^{-m-\frac{1}{4}}\sum_{\mu |k ;
\mu>\sqrt{k}}\mu^{-1}\leq T^{\frac{3}{4}+2\gamma
m+\alpha(\frac{1}{4}-m+\delta)}.
\end{eqnarray}
Therefore, to have this tail bounded by $T^{\frac{3}{4}}$ we shall
choose $\alpha>\frac{2\gamma m}{m-1/4-\delta}$ and since we can
choose $m$ as large as we please, this inequality holds if we assume
$\alpha>2 \gamma$.
\end{proof}

To evaluate the third moment, we have
\begin{eqnarray}
\int_1^T F^{\epsilon}_H (t)^3 dt&=& \sum_{\mu_j,
\nu_j}(-1)^{\nu_1+\nu_2+\nu_3}\prod_{j=1}^3\left\{
\mu_j^{-\frac{5}{4}}\nu_j^{-\frac{1}{4}}
\widehat{\rho}_\epsilon(\frac{\mu_j+\nu_j}{2},\nu_j)\right\}
\nonumber\\&&\quad\quad\quad\times \int_1^T t^{9/4}
\prod_{j=1}^3\cos(2\pi\sqrt{t}\sqrt{\mu_j\nu_j}-\frac{\pi}{4})
 dt\nonumber\\
&=& \frac{1}{8}\sum_{\mu_j,
\nu_j}(-1)^{\nu_1+\nu_2+\nu_3}\prod_{j=1}^3\left\{
\mu_j^{-\frac{5}{4}}\nu_j^{-\frac{1}{4}}
\widehat{\rho}_\epsilon(\frac{\mu_j+\nu_j}{2},\nu_j)\right\}\nonumber\\
&&\quad\quad\quad\times \int_1^T t^{9/4} e^{\pi i\left(\pm(
2\sqrt{ t\mu_1\nu_1}-\frac{1}{4})\pm(2\sqrt{t
\mu_2\nu_2}-\frac{1}{4}) \pm(2\sqrt{t
\mu_3\nu_3}-\frac{1}{4})\right)} dt, \nonumber\end{eqnarray} where
$\pm$ means that we have a total of 8 terms, one for each choice
of $+$ or $-$ sign. Next, we show that all the indices for which
$\pm( 2\sqrt{ t\mu_1\nu_1}-\frac{1}{4})\pm(2\sqrt{t
\mu_2\nu_2}-\frac{1}{4}) \pm(2\sqrt{t \mu_3\nu_3}-\frac{1}{4})\neq
0$ lead to lower order terms. Without loss of generality, we
continue the proof by considering the following summation

\begin{align}S:=\sum_{\stackrel{\Delta\neq 0}{
\mu_1\nu_1\geq\mu_2\nu_2}}(-1)^{\nu_1+\nu_2+\nu_3}\prod_{j=1}^3\left\{
\mu_j^{-\frac{5}{4}}\nu_j^{-\frac{1}{4}}
\widehat{\rho}_\epsilon(\frac{\mu_j+\nu_j}{2},\nu_j)\right\}
\int_1^T t^{9/4} e^{2\pi i\sqrt{t}\,\Delta-\frac{\pi i}{4}}
dt,\label{sam1}\end{align} where $\Delta:=\sqrt{\mu_1\nu_1}+\sqrt{
\mu_2\nu_2}-\sqrt{ \mu_3\nu_3}$. Since $|\widehat{\rho}_\epsilon|$
is bounded above by $1$, we have

\begin{eqnarray*}
|S|\leq \sum_{\stackrel{\Delta\neq 0}{
\mu_1\nu_1\geq\mu_2\nu_2}}\mu_1^{-\frac{5}{4}}\nu_1^
{-\frac{1}{4}}\mu_2^{-\frac{5}{4}}
\nu_2^{-\frac{1}{4}}\mu_3^{-\frac{5}{4}}\nu_3^{-\frac{1}{4}}
\left| \int_1^T t^{9/4} e^{2\pi i\sqrt{t}\,
\Delta-\frac{\pi i}{4}} dt\right|.\label{sam2}\\
\nonumber\end{eqnarray*}
For fixed positive $\sigma$ and $\beta$
to be specified later, break the summation in ~(\ref{sam2}) in
three parts.

\vspace{.2in}

\noindent\emph{Case 1}: If $|\Delta | > (\mu_1\nu_1)^{1/2-\sigma}$
then using the integral estimate \begin{equation*}\left|\int_1^T
t^{9/4}e^{i \omega \sqrt{t}}dt\right|\leq 2\left
[\frac{t^{11/4}}{|\omega|}\right]_1^T \ll
\frac{T^{11/4}}{|\omega|},\end{equation*} we have

\begin{eqnarray*}
S_1:=\sum_{\stackrel{|\Delta|>(\mu_1\nu_1)^{1/2-\sigma}}{\stackrel{
\mu_1\nu_1\geq\mu_2\nu_2 \, ,\,
\mu_j>\nu_j>0}{\mu_j\nu_j<T^\alpha}}}\mu_1^{-\frac{5}{4}}\nu_1^{-\frac{1}{4}}
\mu_2^{-\frac{5}{4}}\nu_2^{-\frac{1}{4}}\mu_3^{-\frac{5}{4}}\nu_3^{-\frac{1}{4}}
\left| \int_1^T t^{9/4} e^{2\pi i\sqrt{t}\, \Delta-\frac{\pi
i}{4}} dt\right|\end{eqnarray*}
\begin{eqnarray}&\ll& T^{\frac{11}{4}}
\sum_{\stackrel{|\Delta|>(\mu_1\nu_1)^{1/2-\sigma}}{\stackrel{
\mu_1\nu_1\geq\mu_2\nu_2 \, ,\,
\mu_j>\nu_j>0}{\mu_j\nu_j<T^\alpha}}}
\mu_1^{-\frac{5}{4}}\nu_1^{-\frac{1}{4}}\mu_2^{-\frac{5}{4}}
\nu_2^{-\frac{1}{4}}\mu_3^{-\frac{5}{4}}\nu_3^{-\frac{1}{4}}\left|
\Delta\right|^{-1} \nonumber\end{eqnarray}
Applying the condition
on $\Delta$ we find

\begin{eqnarray}
S_1&\ll& T^{\frac{11}{4}} \sum_{\stackrel{
\mu_1\nu_1\geq\mu_2\nu_2 \, ,\,
\mu_j>\nu_j>0}{\mu_j\nu_j<T^\alpha}}
\mu_1^{-\frac{5}{4}}\nu_1^{-\frac{1}{4}}\mu_2^{-\frac{5}{4}}\nu_2^
{-\frac{1}{4}}\mu_3^{-\frac{5}{4}}\nu_3^{-\frac{1}{4}}(\mu_1\nu_1)^{-1/2+\sigma}
\nonumber\\
&= &T^{\frac{11}{4}}\sum_{\stackrel{ \mu_1\nu_1\geq\mu_2\nu_2 \, ,\,
\mu_j>\nu_j>0}{\mu_j\nu_j<T^\alpha}}
(\mu_1\nu_1)^{-3/4+\sigma}\mu_1^{-1}(\mu_2\nu_2)^{-1/4}\mu_2^{-1}(\mu_3\nu_3)^{-1/4}\mu_3^{-1}
\nonumber\\
&\ll& T^{\frac{11}{4}}\sum_{\stackrel{
\mu_j>\nu_j>0}{\mu_j\nu_j<T^\alpha}}
(\mu_1\nu_1)^{-1/2+\sigma+\delta}\mu_1^{-1}((\mu_3\nu_3)^{-1/4}\mu_3^{-1}
\label{11s1}\\
&\ll& T^{\frac{11}{4}+\delta}\sum_{0<m_j<T^\alpha}
m_1^{-1+\sigma}m_3^{-3/4}.\label{1s1}\end{eqnarray} To obtain
~(\ref{11s1}) and ~(\ref{1s1}) we have used the fact that

\begin{equation*}\sum_{\mu_j>\nu_j>0}(\mu_j\nu_j)^{-1/4}\mu_j^{-1}\ll\sum_{m_j}
m_j^{-3/4+\delta}\end{equation*} for $m_j:=\mu_j\nu_j$  and
arbitrary $\delta>0$.

Therefore,
\begin{equation}
 S_1=O\left( T^{11/4+\alpha/4+\delta+\sigma\alpha}\right),
\label{1}
\end{equation}
and to have $S_1=o(T^{13/4})$ we need the condition \begin{equation}\alpha(\frac{1}{4}+\sigma)
<\frac{1}{2}\label{1'}\end{equation} on
$\alpha$ and $\sigma$.

\vspace{.2in}

\noindent \emph{Case 2:} If
$|\Delta|\leq(\mu_1\nu_1)^{1/2-\sigma}$, then we can prove that
$\mu_3\nu_3$ has basically the same order of magnitude as
$\mu_1\nu_1$ and the number of the solutions for $\mu_3\nu_3$
satisfying $|\Delta| \leq(\mu_1\nu_1)^{1/2-\sigma}$ is bounded by
$ 1+5 |\Delta| \sqrt{\mu_1\nu_1}$. To prove these claims we note
that
\begin{equation*}
\mu_3\nu_3=\left(\sqrt{\mu_1\nu_1}+\sqrt{\mu_2\nu_2}\right)^2+\Delta^2\pm
2\Delta\left(\sqrt{\mu_1\nu_1}+\sqrt{\mu_2\nu_2}\right).
\end{equation*}
Therefore,

\begin{equation}\left|\mu_3\nu_3-\left(\sqrt{\mu_1\nu_1}+
\sqrt{\mu_2\nu_2}\right)^2\right|\leq
\Delta^2+4|\Delta|\sqrt{\mu_1\nu_1}\leq
5\left({\mu_1\nu_1}\right)^{1-\sigma},\label{2}
\end{equation}
which shows the first claim is true

\begin{equation}\frac{\mu_1\nu_1}{2}\leq\mu_3\nu_3\leq 9
\mu_1\nu_1.
\end{equation}
The second claim is also clear by looking
at ~(\ref{2}) which can be written
as

\begin{equation}\left|\mu_3\nu_3-\left(\sqrt{\mu_1\nu_1}+
\sqrt{\mu_2\nu_2}\right)^2\right|\leq
5|\Delta|\sqrt{\mu_1\nu_1}.\nonumber
\end{equation}
To consider the case when $|\Delta|\leq(\mu_1\nu_1)^{1/2-\sigma}$,
we divide it to two subcases: the first one is if $\Delta$ is
permitted to be very small, i.e. $|\Delta|\leq T^{-\beta}$. The
second case is if $ T^{-\beta}<|\Delta|\leq
(\mu_1\nu_1)^{1/2-\sigma}.$

\vspace{.2in}

\noindent \emph{Subcase 2.1:} If $|\Delta|\leq T^{-\beta}$, then
by using the trivial bound on the integral we find

\begin{eqnarray}S_2&:=&\sum_{\stackrel{0<|\Delta|\leq T^{-\beta}}{\stackrel{
\mu_1\nu_1\geq\mu_2\nu_2 \, ,\,
\mu_j>\nu_j>0}{\mu_j\nu_j<T^\alpha}}}\mu_1^{-\frac{5}{4}}\nu_1^
{-\frac{1}{4}}\mu_2^{-\frac{5}{4}}\nu_2^
{-\frac{1}{4}}\mu_3^{-\frac{5}{4}}\nu_3^{-\frac{1}{4}} \left|
\int_1^T t^{9/4} e^{2\pi i\sqrt{t}\, \Delta -\frac{\pi i}{4}}
dt\right| \nonumber\\ &\ll& T^{\frac{13}{4}}
\sum_{\stackrel{0<|\Delta|\leq T^{-\beta}}{\stackrel{
\mu_1\nu_1\geq\mu_2\nu_2 \, ,\,
\mu_j>\nu_j>0}{\mu_j\nu_j<T^\alpha}}}
\mu_1^{-\frac{5}{4}}\nu_1^{-\frac{1}{4}}\mu_2^{-\frac{5}{4}}
\nu_2^{-\frac{1}{4}}\mu_3^{-\frac{5}{4}}\nu_3^{-\frac{1}{4}} \, .
\nonumber\end{eqnarray} Using the fact that $\mu_3\nu_3$ has
basically the same order of magnitude as $\mu_1\nu_1$ and the
number of the solutions for $\mu_3\nu_3$ satisfying $|\Delta|
\leq(\mu_1\nu_1)^{1/2-\sigma}$ is bounded by $ 1+5 |\Delta|
\sqrt{\mu_1\nu_1}$ we get that

\begin{eqnarray}
&S_2&\ll T^{\frac{13}{4}} \sum_{\stackrel{\stackrel{0<|\Delta|\leq
T^{-\beta}}{\mu_1\nu_1\geq\mu_2\nu_2}}{\stackrel{\mu_j>\nu_j>0}{\mu_j\nu_j<T^\alpha}}}
\mu_1^{-\frac{5}{4}}\nu_1^{-\frac{1}{4}}\mu_2^{-\frac{5}{4}}
\nu_2^{-\frac{1}{4}}\left(\frac{\mu_1\nu_1}{2}\right)^{-\frac{3}{4}+\delta}
\left(1+5 |\Delta| \sqrt{\mu_1\nu_1} \right)
\nonumber\\
&\ll& T^{\frac{13}{4}} \sum_{\stackrel{\stackrel{0<|\Delta|\leq
T^{-\beta}}{\mu_1\nu_1\geq\mu_2\nu_2}}{\stackrel{\mu_j>\nu_j>0}{\mu_j\nu_j<T^\alpha}}}
(\mu_1\nu_1)^{-1+\delta}\mu_1^{-1}\mu_2^{-\frac{5}{4}}\nu_2^{-\frac{1}{4}}+
5T^{\frac{13}{4}-\beta}\sum_{{\stackrel{\stackrel{
\mu_1\nu_1\geq\mu_2\nu_2\,}{
\mu_j>\nu_j>0}}{\mu_j\nu_j<T^\alpha}}}(\mu_1\nu_1)^
{-\frac{1}{2}+\delta}\mu_1^{-1}\mu_2^{-\frac{5}{4}}\nu_2^{-\frac{1}{4}}
\nonumber\\
&\ll& T^{\frac{13}{4}} \sum_{{\stackrel{
\mu_1>\nu_1>0}{T^{2\beta/3}\leq\mu_1\nu_1<T^\alpha}}}
(\mu_1\nu_1)^{-\frac{3}{4}+\delta}\mu_1^{-1}+
5T^{\frac{13}{4}-\beta}\sum_{{\stackrel{
\mu_1>\nu_1>0}{\mu_1\nu_1<T^\alpha}}}(\mu_1\nu_1)^{-\frac{1}{4}+\delta}\mu_1^{-1}
\nonumber\\
&\ll& T^{\frac{13}{4}} \sum_{T^{2\beta/3}\leq
m_1}m_1^{-\frac{5}{4}+\delta}+
5T^{\frac{13}{4}-\beta}\sum_{0<m_1<T^\alpha}m_1^{-\frac{3}{4}+\delta}
\nonumber \, .
\end{eqnarray} Therefore,

\begin{equation}S_2=O \left(T^{\frac{13}{4}-\frac{\beta}{6}+\delta}+
T^{\frac{13}{4}-\beta+\frac{\alpha}{4}+\delta}\right),\label{3}
\end{equation}
and to have $S_2=o(T^{\frac{13}{4}})$ we need the
condition
\begin{equation}
\frac{\alpha}{4}-\beta<0.\label{2'}
\end{equation}
on $\alpha$ and $\beta$ to be satisfied.

\vspace{.2in}

\noindent \emph{Subcase 2.2:} Finally let us consider the last
case, that is when
$T^{-\beta}<|\Delta|\leq(\mu_1\nu_1)^{1/2-\sigma}$. We have

\begin{eqnarray*}
S_3&:=&\sum_{\stackrel{T^{-\beta}<|\Delta|\leq(\mu_1\nu_1)^{1/2-\sigma}}{\stackrel{
\mu_1\nu_1\geq\mu_2\nu_2 \, ,\,
\mu_j>\nu_j>0}{\mu_j\nu_j<T^\alpha}}}\mu_1^{-\frac{5}{4}}\nu_1^
{-\frac{1}{4}}\mu_2^{-\frac{5}{4}}\nu_2^{-\frac{1}{4}}\mu_3^{-\frac{5}{4}}\nu_3^{-\frac{1}{4}}
\left| \int_1^T t^{9/4} e^{2\pi i\sqrt{t}\, \Delta -\frac{\pi
i}{4}} dt\right|\\&\ll& T^{\frac{11}{4}}
\sum_{\stackrel{T^{-\beta}<|\Delta|\leq(\mu_1\nu_1)^{1/2-\sigma}}{\stackrel{
\mu_1\nu_1\geq\mu_2\nu_2 \, ,\,
\mu_j>\nu_j>0}{\mu_j\nu_j<T^\alpha}}}
\mu_1^{-\frac{5}{4}}\nu_1^{-\frac{1}{4}}\mu_2^{-\frac{5}{4}}
\nu_2^{-\frac{1}{4}}\mu_3^{-\frac{5}{4}}\nu_3^{-\frac{1}{4}}\left|
\Delta\right|^{-1} \nonumber   \, .
\end{eqnarray*} Like before,
we use the fact that $\mu_3\nu_3$ has basically the same order of
magnitude as $\mu_1\nu_1$ and the number of the solutions for
$\mu_3\nu_3$ satisfying $|\Delta| \leq(\mu_1\nu_1)^{1/2-\sigma}$
is bounded by $ 1+5 |\Delta| \sqrt{\mu_1\nu_1}$ to write

\begin{eqnarray}
S_3&\ll& T^{\frac{11}{4}}
\sum_{\stackrel{T^{-\beta}<|\Delta|\leq(\mu_1\nu_1)^{1/2-\sigma}}
{\stackrel{ \mu_1\nu_1\geq\mu_2\nu_2 \, ,\,
\mu_j>\nu_j>0}{\mu_j\nu_j<T^\alpha}}}
\mu_1^{-\frac{5}{4}}\nu_1^{-\frac{1}{4}}\mu_2^{-\frac{5}{4}}
\nu_2^{-\frac{1}{4}}\left(\frac{\mu_1\nu_1}{2}\right)^{-\frac{3}{4}+\delta}
\left(1+5 |\Delta| \sqrt{\mu_1\nu_1} \right)|\Delta|^{-1}
\nonumber\\
&\ll&
T^{\frac{11}{4}}\sum_{\stackrel{T^{-\beta}<|\Delta|\leq(\mu_1\nu_1)^{1/2-\sigma}}{\stackrel{
\mu_1\nu_1\geq\mu_2\nu_2 \, ,\,
\mu_j>\nu_j>0}{\mu_j\nu_j<T^\alpha}}}
(\mu_1\nu_1)^{-1+\delta}\mu_1^{-1}(\mu_2\nu_2)^{-1/4}\mu_2^{-1}|\Delta|^{-1}
\nonumber\\&& + \quad T^{\frac{11}{4}}\sum_{\stackrel{
\mu_1\nu_1\geq\mu_2\nu_2 \, ,\,
\mu_j>\nu_j>0}{\mu_j\nu_j<T^\alpha}}
(\mu_1\nu_1)^{-1/2+\delta}\mu_1^{-1}(\mu_2\nu_2)^{-1/4}\mu_2^{-1}
\nonumber\\
&\ll& T^{\frac{11}{4}+\beta}\sum_{\stackrel{
\mu_1>\nu_1>0}{\mu_1\nu_1<T^\alpha}}
(\mu_1\nu_1)^{-3/4+\delta}\mu_1^{-1}+T^{\frac{11}{4}}\sum_{\stackrel{
\mu_1>\nu_1>0}{\mu_1\nu_1<T^\alpha}}
(\mu_1\nu_1)^{-1/4+\delta}\mu_1^{-1}. \nonumber\end{eqnarray}
Therefore
\begin{eqnarray}
S_3=O\left(T^{\frac{11}{4}+\beta+\delta}+T^{11/4+\alpha/4+\delta}\right),
\label{4}\end{eqnarray}  and to have $S_3=o(T^{\frac{13}{4}})$ we
require

\begin{equation}\alpha<2 \quad\hbox{ and }\quad
\beta<\frac{1}{2}.\label{3'}
\end{equation}
Therefore, taking all
the conditions from Remark \ref{rem2}, Lemma \ref{lem},
~(\ref{1'}), ~(\ref{2'}) and ~(\ref{3'}) together, we have proved
that by making an arbitrary choice for $\alpha$, $\beta$ and
$\sigma$ satisfying

\begin{equation}\frac{3}{2}<\alpha<4\beta<2 \quad\hbox{ and
}\quad 0<\sigma<\frac{1}{2\alpha}-\frac{1}{4},
\end{equation} we
have

\begin{align}
\int_1^T F^{\epsilon}_H (t)^3 dt&=&
\frac{3}{8}\sum_{\Delta=0}(-1)^{\nu_1+\nu_2+\nu_3}\mu_1^
{-\frac{5}{4}}\nu_1^{-\frac{1}{4}}\mu_2^{-\frac{5}{4}}
\nu_2^{-\frac{1}{4}}\mu_3^{-\frac{5}{4}}\nu_3^{-\frac{1}{4}}\int_1^T
2\, t^{9/4} \cos(\frac{\pi}{4}) dt \nonumber\\&&
\widehat{\rho}_\epsilon(\frac{\mu_1+\nu_1}{2},\nu_1)
\widehat{\rho}_\epsilon(\frac{\mu_2+\nu_2}{2},\nu_2)
\widehat{\rho}_\epsilon(\frac{\mu_3+\nu_3}{2},\nu_3)
+O(|S|)\nonumber
\\&=& \frac{3\sqrt{2}}{26}\, T^{ 13/4}
\sum_{\Delta=0}(-1)^{\nu_1+\nu_2+\nu_3}\mu_1^{-\frac{5}{4}}\nu_1^{-\frac{1}{4}}
\mu_2^{-\frac{5}{4}}\nu_2^{-\frac{1}{4}}\mu_3^{-\frac{5}{4}}\nu_3^{-\frac{1}{4}}
\quad\quad\quad\quad\quad\quad
\label{F}\\&&\widehat{\rho}_\epsilon(\frac{\mu_1+\nu_1}{2},\nu_1)
\widehat{\rho}_\epsilon(\frac{\mu_2+\nu_2}{2},\nu_2)
\widehat{\rho}_\epsilon(\frac{\mu_3+\nu_3}{2},\nu_3)+
O(|S|),\nonumber
\end{align}
where $O(|S|)=O(S_1+S_2+S_3)=o(T^{13/4})$. Next we split the
summation in ~(\ref{F}) into the pieces where $\mu_3< T^{1/4}$ and
$\mu_3 \geq T^{1/4}$. We claim that the piece where $\mu_3 \geq
T^{1/4}$ is residual. To see this, note that if $
\sqrt{\mu_1\nu_1}+\sqrt{ \mu_2\nu_2}=\sqrt{ \mu_3\nu_3}$ then
there exists integers $k$, $m_1$, $m_2$ and $m_3$ such that
$\mu_j\nu_j=k m_j^2$ and $m_1+m_2=m_3$. Therefore,
\begin{align}
&&T^{ \frac{13}{4}} \sum_{ \stackrel{\stackrel{0<\nu_j<\mu_j;}{
\nu_j\equiv\mu_j \imod{2}}}{\Delta=0; \: \mu_3\geq
T^{1/4}}}\mu_1^{-\frac{5}{4}}\nu_1^{-\frac{1}{4}}
\mu_2^{-\frac{5}{4}}\nu_2^{-\frac{1}{4}}\mu_3^{-\frac{5}{4}}\nu_3^{-\frac{1}{4}}
\leq T^3 \sum_{ \stackrel{0<\nu_j<\mu_j}{\Delta=0
}}\mu_1^{-\frac{5}{4}}\nu_1^{-\frac{1}{4}}
\mu_2^{-\frac{5}{4}}\nu_2^{-\frac{1}{4}}\mu_3^{-\frac{1}{4}}\nu_3^{-\frac{1}{4}}
\nonumber\\&&\leq T^3 \sum_{m_1>0; m_2>0;
k>0}k^{-\frac{3}{4}}m_1^{-\frac{1}{2}}m_2^{-\frac{1}{2}}(m_1+m_2)^{-\frac{1}{2}}
\sum_{\mu_j|km_j^2; \mu_j>k^{1/2}m_j}\mu_1^{-1}\mu_2^{-1}\nonumber\\
&&\leq T^3 \sum_{m_1>0; m_2>0;
k>0}k^{-\frac{7}{4}}m_1^{-\frac{3}{2}}m_2^{-\frac{3}{2}}(m_1+m_2)^{-\frac{1}{2}}
d(k m_1^2)d(k m_2^2)\ll T^3.\label{res}
\end{align}
On the other hand, if $\mu_3<T^{1/4}$ then for $j=1, 2$ we have
$\mu_j\leq \mu_j\nu_j\leq \mu_3\nu_3\leq T^{1/2}$. Since
$\epsilon=T^{-\gamma}\ll T^{-3/4}$, we have that $\epsilon\nu_j<
\epsilon\mu_j< T^{-1/4}$ for $j=1, 2, 3$. Therefore, we
expand the functions
$\widehat{\rho}_\epsilon(\frac{\mu_j+\nu_j}{2},\nu_j)$ in Taylor series around the
point $(0, 0)$ and use ~(\ref{res}), so that we can evaluate the
summation in ~(\ref{F}) as

\begin{align}
\int_1^T F^{\epsilon}_H (t)^3 dt&=&\frac{3\sqrt{2}}{26}\:\:
T^{\frac{13}{4}} \sum_{ \stackrel{\stackrel{0<\nu_j<\mu_j;}{
\nu_j\equiv\mu_j \imod{2}}}{\Delta=0}}(-1)^
{\nu_1+\nu_2+\nu_3}\mu_1^{-\frac{5}{4}}\nu_1^{-\frac{1}{4}}
\mu_2^{-\frac{5}{4}}\nu_2^{-\frac{1}{4}}\mu_3^{-\frac{5}{4}}\nu_3^{-\frac{1}{4}}
+ O(|S|).\nonumber
\end{align}
Repeating a similar argument for the second summation in
$R_H^\epsilon(t)$ given in ~(\ref{repeat}), we have proved that

\begin{eqnarray}
\int_1^T R^{\epsilon}_H (t)^3 dt&=& b_3\, T^{ 13/4}
+O(|S|),\label{cc}
\end{eqnarray}
where $b_3$ is the constant defined by

\begin{eqnarray}
b_3&=&\frac{3\sqrt{2}}{26\pi^3} \sum_{ \stackrel{0<\nu_j<\mu_j;\:
{\nu_j\equiv\mu_j\imod{2}}}{\sqrt{\mu_1\nu_1}+\sqrt{
\mu_2\nu_2}=\sqrt{
\mu_3\nu_3}}}(-1)^{\nu_1+\nu_2+\nu_3}\mu_1^{-\frac{5}{4}}\nu_1^{-\frac{1}{4}}
\mu_2^{-\frac{5}{4}}\nu_2^{-\frac{1}{4}}\mu_3^{-\frac{5}{4}}\nu_3^{-\frac{1}{4}}
\nonumber\\&& +\frac{3\sqrt{2}}{208\pi^3} \sum_{ \stackrel{0<\nu_j
}{\nu_1+\nu_2=\nu_3}}(-1)^{\nu_1+\nu_2+\nu_3}\nu_1^{-\frac{3}{2}}
\nu_2^{-\frac{3}{2}}\nu_3^{-\frac{3}{2}} \nonumber\\&&
+\frac{9\sqrt{2}}{52\pi^3}\sum_{ \stackrel{0<\nu_j<\mu_j;\:
\nu_j\equiv\mu_j \imod{2}}{\sqrt{\mu_1\nu_1}+\sqrt{
\mu_2\nu_2}=\nu_3}}(-1)^{\nu_1+\nu_2+\nu_3}
\mu_1^{-\frac{5}{4}}\nu_1^{-\frac{1}{4}}
\mu_2^{-\frac{5}{4}}\nu_2^{-\frac{1}{4}}\nu_3^{-\frac{3}{2}}
\nonumber\\&&+ \frac{9\sqrt{2}}{104\pi^3} \sum_{
\stackrel{0<\nu_j<\mu_j;\:
\nu_j\equiv\mu_j\imod{2}}{\sqrt{\mu_3\nu_3}=\nu_1+\nu_2}}(-1)^{\nu_1+\nu_2+\nu_3}
\nu_1^{-\frac{3}{2}}\mu_2^{-\frac{5}{4}}\nu_2^{-\frac{1}{4}}
\mu_3^{-\frac{5}{4}}\nu_3^{-\frac{1}{4}}.\label{c}
\end{eqnarray}
Now to prove that $b_3$ is a positive constant, we show that every
summation on the right hand side of ~(\ref{c}) is positive. Since
every one of these summations is a
 special case of the first sum where one impose an extra condition that
$\mu_j=\nu_j$ for one or two of $j$s, we may restrict our
attention only at the first sum. From
$\sqrt(\mu_1\nu_1)=\sqrt(\mu_2\nu_2)+\sqrt(\mu_3\nu_3)$, we get
that for some square free $k$ and integers $m_1, m_2$ and $m_3$
satisfying $m_1+m_2=m_3$, we have $\mu_j\nu_j=m_j^2 k$. From the
conditions that $\nu_j\equiv\mu_j \imod{2}$ and also that $k$ is
square free, we can see that $\nu_j$ and $m_j$ should have the
same parity. Therefore from $m_1+m_2=m_3$, we get
$(-1)^{\nu_1+\nu_2+\nu_3}=1$. This proves that the first series is
positive and similarly the other three are also positive. So the
constant $b_3$ is strictly positive.

The last step in the proof of the Theorem \ref{thm} is to use
Lemma \ref{lem1} to get rid of the mollification in
$\rho_\epsilon$ and prove the third moment estimate for $R_H(t)$,
which is the remainder term corresponding to type II eigenvalues.
From Lemma \ref{lem1} by choosing $\epsilon=T^{-\gamma}$ and
$1<t<T$ for $\gamma>3/4$ we find

\begin{eqnarray}
N^{\epsilon}_H(t-c_\gamma T^{1-\gamma})\leq N_H(2\pi t)\leq
N^{\epsilon}_H (t+c_\gamma T^{1-\gamma}).\label{11}
\end{eqnarray}
From Proposition \ref{prop1} we have
\begin{eqnarray}N^{\epsilon}_H(t \,_{-}^{+}\, c_\gamma T^{1-\gamma})=
\frac{2}{3}t^{\frac{3}{2}}-\frac{1}{2}t+R^{\epsilon}_H(t
\,_{-}^{+}\, c_\gamma T^{1-\gamma}) +O(T^{\frac{3}{2}-\gamma}).
\label{22}\end{eqnarray}
Therefore, from ~(\ref{11}) and
~(\ref{22}) we find

\begin{equation}
R^{\epsilon}_H(t-c_\gamma
T^{1-\gamma})+O(T^{\frac{3}{2}-\gamma})\leq R_H(2\pi t)\leq
R^{\epsilon}_H (t+c_\gamma
T^{1-\gamma})+O(T^{\frac{3}{2}-\gamma}),\label{33}
\end{equation}
where

\begin{eqnarray}
R_H(2\pi t):=N_H(2\pi t)-\frac{2}{3}t^{\frac{3}{2}}+\frac{1}{2}t.
\end{eqnarray}
Taking the third moment of the right hand side of ~(\ref{33}), we
obtain

\begin{eqnarray}
\int_1^T R_H(2\pi t)^3 dt&\leq& \int_1^T R^{\epsilon}_H
(t+c_\gamma T^{1-\gamma})^3 dt+O(T^{\frac{3}{2}-\gamma})\int_1^T
R^{\epsilon}_H (t+c_\gamma T^{1-\gamma})^2
dt\nonumber\\&&+O(T^{3-2\gamma})\int_1^T R^{\epsilon}_H
(t+c_\gamma T^{1-\gamma})dt+O(T^{\frac{11}{2}-3\gamma}).\label{44}
\end{eqnarray}
Next we use the second moment result proved in \cite{KT} which
states that
\begin{eqnarray}
\int_1^T {R^{\epsilon}_H(t)}^2 dt =d_2
T^{\frac{5}{2}}+o(T^{\frac{5}{2}}),\label{55}
\end{eqnarray}
and, applying a simple H\"{o}lder inequality, shows that
\begin{eqnarray}
\int_1^T |{R^{\epsilon}_H(t)}| dt =O( T^{\frac{7}{4}}).\label{66}
\end{eqnarray}
Applying the results from ~(\ref{55}) and ~(\ref{66}) back in
~(\ref{44}) proves that
\begin{eqnarray}
\int_1^T R_H(2\pi t)^3 dt&\leq& \int_1^T R^{\epsilon}_H
(t+c_\gamma T^{1-\gamma})^3
dt+O(T^{4-\gamma})+O(T^{\frac{19}{4}-2\gamma})
+O(T^{\frac{11}{2}-3\gamma})\nonumber\\&\leq& \int_{1+c_\gamma
T^{1-\gamma}}^{T+c_\gamma T^{1-\gamma}} R^{\epsilon}_H (t)^3
dt+O(T^{4-\gamma}).\nonumber
\end{eqnarray}
Finally, we use the result from ~(\ref{cc}) and prove that
\begin{eqnarray}
\int_1^T R_H(2\pi t)^3 dt&\leq& b_3
T^{\frac{13}{4}}+O(|S|)+O(T^{4-\gamma}).\label{rh1}
\end{eqnarray}
Similarly, from the first inequality in ~(\ref{33}) we find

\begin{eqnarray}
\int_1^T R_H(2\pi t)^3 dt&\geq& b_3
T^{\frac{13}{4}}+O(|S|)+O(T^{4-\gamma}).\label{rh2}
\end{eqnarray}
From ~(\ref{rh1}) and ~(\ref{rh2}) and using $|S|\leq S_1+S_2+S_3$
we have

\begin{eqnarray}
\int_1^T R_H(2\pi t)^3 dt= b_3\,
T^{\frac{13}{4}}+O(S_1)+O(S_2)+O(S_3)+O(T^{4-\gamma}).\label{rh}
\end{eqnarray}
Solving an optimization problem on the parameters $\gamma$,
$\alpha$, $\beta$ and $\sigma$ satisfying the conditions

\begin{equation}\frac{3}{2}<2\gamma<\alpha<4\beta<2 \quad\hbox{ and
}\quad 0<\sigma<\frac{1}{2\alpha}-\frac{1}{4},
\end{equation}
we find that for an arbitrary small positive $\delta '<1/7$ and
$\gamma=11/14$, $\alpha=11/7+\delta '$, $\beta=3/7$ and
$\sigma=\delta '/4$, we obtain

\begin{eqnarray}
\int_1^T R_H(2\pi t)^3 dt= b_3\,
T^{\frac{13}{4}}+O_\delta(T^{\frac{13}{4}-\frac{1}{28}+\delta}),
\end{eqnarray}
for any arbitrary small positive $\delta$. This shows that

\begin{eqnarray}
\int_1^T R(t)^3 dt= d_3\,
T^{\frac{13}{4}}+O_\delta(T^{\frac{13}{4}-\frac{1}{28}+\delta}),
\end{eqnarray}
where $d_3=(2\pi)^{-9/4} b_3$. This completes the proof of Theorem
\ref{thm}.


\begin{thebibliography}{ABCD}

\bibitem[BG]{BG}
V. Bentkus, F. G\"{o}tze, \emph{Lattice point problems and distribution
of values of quadratic forms}, Ann. of Math. (2) 50:3 (1999),
977--1027.

\bibitem[B\'{e}]{Be}
P.H. B\'{e}rard, \emph{On the wave equation on a compact Riemannian
manifold without conjugate points}, Math. Z. 155:3 (1977),
249--276.

\bibitem[Bl]{Bl}
P. Bleher, \emph{On the distribution of the number of lattice points
inside a family of convex ovals}, Duke Math. J. 67:3 (1992),
461--481.


\bibitem[Bu]{Bu}  L. Butler, \emph{Integrable geodesic flows on $n$-step
nilmanifolds}, J. Geom. Phys., {\bf 36} (2000), no. 3-4, 315--323.


\bibitem[Co]{Co}
Y. Colin de Verdi\`{e}re, \emph{Spectre conjoint d'op\'{e}rateurs
pseudo-diiff\'{e}rentiels qui commtent. II. Le cas
int\'{e}grable}, Math. Z. 171(1) (1980) 51--73.



\bibitem[CPT]{PT2}
D. Chung, Y.N. Petridis, J.A. Toth, \emph{The remainder in Weyl's
law for Heisenberg manifolds II},  Bonner Mathematische Schriften,
Nr. {\bf 360}, Bonn, 2003, 16 pages.


\bibitem[Cr]{Cr}
H. Cram\'{e}r, \emph{\"{U}ber zwei S\"{a}tze von Herrn G.H.
Hardy}, Math. Z. 15(1922) 201--210.

\bibitem[DG]{DG}
J.J. Duistermaat \& V. Guillemin, \emph{The spectrum of positive
elliptic operators and periodic bicharacteristics}, Invent. Math.
29(1) (1975) 39-79.

\bibitem[Fo]{Fo}
G.B. Folland, Harmonic Analysis in Phase Space, Princeton
University Press(1989) 9--73.

\bibitem[Fr]{Fr}
F. Fricker, Einf\"{u}hrung in die Gitterpunketlehre, [Introduction
to lattice point theory] Lehrb\"{u}cher und Monographien aus dem
Gebiete der Exakten Wissenschaften(LMW), Mathematische Reihe
[Textbooks and Monographs in the Exact Sciences] 73,
Birkh\"{a}user Verlag, Basel-Boston, Mas., 1982.

\bibitem[G\"{o}]{Go}
F. G\"{o}tze, \emph{Lattice point problems and values of quadratic forms},
Invent. Math. 157(1) (2004) 195--226.

\bibitem[GW]{GW}
C. Gordon \& E. Wilson, \emph{The spectrum of the Laplacian on
Reimannian Heisenberg manifolds}, Michigan Math. J. 33(2) (1986)
253--271.


\bibitem[Ha]{Ha}
G.H. Hardy, \emph{On the expression of a number as the sum of two
squares}, Quart. J. Math. 46(1915) 263--283.

\bibitem[HB]{HB} D. R. Heath-Brown, \emph{The distribution and moments of the
error term in the Dirichlet divisor problem}, Acta Arithmetica
4(1992) 389--415.


\bibitem[H\"{o}]{Ho}
L. H\"{o}rmander, \emph{The spectral function of an elliptic
operator}, Acta Math. 121(1968) 193--218.

\bibitem[Hu]{Hu}
M.N. Huxley, \emph{Exponential sums and lattice points III}, Proc. London
Math. Soc. (3) 87(2003), 591--609.

\bibitem[Iv]{Iv}
V.YA. Ivrii, Precise Spectral Asymptotics for elliptic Operators
Acting in Fibrings over Manifolds with Boundary, Springer Lecture
Notes in Mathematics 1100 (1984).

\bibitem[Kh]{K}
M. Khosravi, \emph{Fifth-power moment of the error term in
Dirichlet's divisor problem}, submitted.

\bibitem[KP]{KP}
M. Khosravi, Y.N., Petridis, \emph{The remainder in weyl's law for
n-dimensional Heisenberg manifolds}, Proc. of Amer.
Math. Soc. 133, 3561--3571.

\bibitem[KT]{KT}
M. Khosravi, J.A., Toth, \emph{ Cram\'{e}r's formula for Heisenberg
manifolds}, Annales de l'institut Fourier 55(7) (2005), 2489--2520.


\bibitem[PT]{PT}
Y.N. Petridis, J.A. Toth, \emph{The remainder in Weyl's law for
Heisenberg manifolds}, J. Diff. Geom. 60(2002) 455--483.

\bibitem[So]{Sou}
K. Soundarajan, \emph{Omega results for the divisor and circle
problems}, IMRN 36(2003) 1987--1998.

\bibitem[St]{St}
E.M. Stein, Harmonic Analysis. Princeton University Press(1993)
527--574.

\bibitem[Ts]{Ts}
K. Tsang, \emph{Higher-power moments of $\Delta(x)$, $E(x)$ and
$P(x)$}, Proc. London. Math. Soc. (3) 65(1992) 65--84.

\bibitem[Vo]{Vo}
A.V. Volovoy, \emph{Improved two-term asymptotics for the eigenvalue
distribution function of an elliptic operator on a compact
manifold}, Comm. Partial Differential Equations 15:11 (1990),
1509--1563.

\end{thebibliography}
\end{document}